\documentclass[final]{siamltex}
\usepackage{color}

\newtheorem{remark}{Remark}[section]

\usepackage{times,amsmath,amsbsy,amssymb}
\usepackage{epsfig,graphicx,epstopdf}

\title{Robust residual-based {\it a posteriori} error estimators for mixed finite element methods for fourth order elliptic singularly perturbed problems}

\author{
    Shaohong Du\thanks{
        School of Mathematics and Statistics, Chongqing Jiaotong University, Chongqing 400074, China. E-mail: {duzheyan.student@sina.com}.
    }
    \and
    Runchang Lin\thanks{
        Department of Mathematics and Physics, Texas A\&M International University (TAMIU), Laredo, Texas 78041, USA. E-mail: {rlin@tamiu.edu}.
        This research is partially supported by the US National Science Foundation under grant DMS-1217268, the National Natural Science Foundation of China under grant 11428103, and a University Research Development Award of TAMIU.
    }
    \and
    Zhimin Zhang\thanks{
        Beijing Computational Science Research Center, Beijing 100193, China. E-mail: {zmzhang@csrc.ac.cn}.
        Department of Mathematics, Wayne State University, Detroit, MI 48202, USA. E-mail: {ag7761@wayne.edu}.
        This research is partially supported by the US National Science Foundation through grant DMS-1419040 and by the National Natural Science Foundation of China under grants 11471031, 91430216, and U1530401.
    }
}

\begin{document}

\maketitle
\renewcommand{\thefootnote}{\fnsymbol{footnote}}

\begin{abstract}
We consider mixed finite element approximation of a singularly
perturbed fourth-order elliptic problem with two different boundary
conditions, and present a new measure of the error, whose components
are balanced with respect to the perturbation parameter. Robust
residual-based {\it a posteriori} estimators for the new measure are
obtained, which are achieved via a novel analytical technique based on
an approximation result. Numerical examples are presented to
validate our theory.
\end{abstract}

\begin{it} Key words. \end{it} fourth order elliptic singularly perturbed problems,
mixed finite element methods, a new measure of the error, robust
residual-based {\it a posteriori} error estimators

\begin{it} AMS subject classifications.\end{it} 65N15, 65N30, 65J15

\section {Introduction}
Let $\Omega$ be a bounded polygonal or polyhedral domain with
Lipschitz boundary $\Gamma=\partial\Omega$ in ${\mathbb{R}}^{d},
d=2\ {\rm or}\ 3$. Consider the following fourth-order singularly
perturbed elliptic equation
\begin{equation}\label{PDE1}
\varepsilon^{2}\triangle^{2}u-\triangle u=f\ \ \ {\rm in}\ \ \Omega
\end{equation}
with boundary conditions
\begin{equation}\label{Boundar condition1}
u=\triangle u=0 \ \ \ {\rm on}\ \ \Gamma
\end{equation}
or
\begin{equation}\label{Boundary condition2}
u=\displaystyle\frac{\partial u}{\partial{\bf n}}=0 \ \ \ {\rm on}\
\ \Gamma,
\end{equation}
where $f\in L^{2}(\Omega)$, $\triangle$ is the standard Laplace
operator, and $\frac{\partial}{\partial{\bf n}}$ denotes the outer
normal derivative on $\Gamma$. In two dimensional cases, the boundary value
problems (\ref{PDE1})-(\ref{Boundar condition1}) and
(\ref{PDE1})-(\ref{Boundary condition2}) arise in the context of
linear elasticity of thin bucking plate with $u$ representing the
displacement of the plate. The dimensionless positive parameter
$\varepsilon$, assumed to be small (i.e., $\varepsilon\ll1$), is
defined by
\begin{equation*}
\varepsilon=\displaystyle\frac{t^{3}E}{12(1-\nu^{2})l^{2}T},
\end{equation*}
where, $t$ is the thickness of the plate, $E$ is the Young modulus
of the elastic material, $\nu$ is the Poisson ratio, $l$ is the
characteristic diameter of the plate, and $T$ is the absolute value
of the density of the isotropic stretching force applied at the end
of the plate \cite{FAN}. In three dimensions, problems
(\ref{PDE1})-(\ref{Boundar condition1}) and
(\ref{PDE1})-(\ref{Boundary condition2}) can be a gross
simplification of the stationary Cahn-Hilliard equations with
$\varepsilon$ being the length of the transition region of phase
separation.

Conforming, nonconforming, and mixed finite element methods for fourth order problem have been extensively studied
\cite{AWANOU, BABUSKA, BGS10, BREZZI, DANUMJAYA, FALK, MORLEY, NILSSEN, SEMPER, SHI, WANG, WANG07, WANG06}. However, its {\em a posteriori} error estimation is a much less explored topic. Even for the Kirchhoff plate bending
problem, the finite element {\it a posteriori} error analysis is
still in its infancy. In 2007, Beir$\tilde{\rm a}$o et al.
\cite{BEIRAO} developed an estimator for the Morley element
approximation using the standard technique for nonconforming
element. Later, Hu et al. \cite{HU} improved the methods of
\cite{BEIRAO,WANGZ,WANG08} by dropping two edge jump terms in both
the energy norm of the error and the estimator, and by dropping the
normal component in the estimators of \cite{BEIRAO,WANGZ}.
Therefore, a naive extension of the estimators in
\cite{BEIRAO,WANGZ,WANG08} to the current problem may probably not
be robust in the parameter $\varepsilon$.

Designing robust {\it a posteriori} estimators is challenging,
especially for singularly perturbed problems, since constants
occurring in estimators usually depend on the small perturbation
parameter $\varepsilon$. This motivates us to think about the
question: What method and norm are suitable for the singularly
perturbed fourth-order elliptic problem? In the literature,
$(\varepsilon^{2}\|\triangle u\|_{L^{2}(\Omega)}^{2}+\|\nabla
u\|_{L^{2}(\Omega)}^{2})^{1/2}$ is a widely used measure for the
primal weak formulation. We recall {\it a priori} estimates in
\cite{Guzman} for boundary condition (\ref{Boundary condition2}) and
convex domain $\Omega$:
\begin{equation}\label{maysp}
\|u\|_{H^{s}(\Omega)}\leq \displaystyle
C\varepsilon^{\frac{3}{2}-s}\|f\|_{L^{2}(\Omega)},\ \ \ {\rm for}\
s=2,3.
\end{equation}
Hereafter, we use $C>0$ to denote a generic constant independent of
$\varepsilon$ with different value at different occurrence. This
leads to
\begin{equation*}
\varepsilon\|\triangle u\|_{L^{2}(\Omega)}\leq
C\varepsilon\|u\|_{H^{2}(\Omega)}\leq
C\varepsilon^{1/2}\|f\|_{L^{2}(\Omega)}.
\end{equation*}

Multiply both sides of (\ref{PDE1}) by $u$, and then integrate over
$\Omega$. Using integration by parts and boundary condition
(\ref{Boundar condition1}) or (\ref{Boundary condition2}), we have
from the Poincar\'{e} inequality that
\begin{equation*}
\begin{array}{lll}
\varepsilon^{2}(\triangle u,\triangle u)+(\nabla u,\nabla
u)&=&(f,u)\leq\|f\|_{L^{2}(\Omega)}\|u\|_{L^{2}(\Omega)}\vspace{2mm}\\
&\leq&C\|f\|_{L^{2}(\Omega)}\|\nabla u\|_{L^{2}(\Omega)}.
\end{array}
\end{equation*}
As a consequence, $\|\nabla u\|_{L^{2}(\Omega)}\leq
C\|f\|_{L^{2}(\Omega)}$. This suggests that the two components of
$\varepsilon\|\triangle u\|_{L^{2}(\Omega)}+\|\nabla
u\|_{L^{2}(\Omega)}$ are unbalanced with respect to $\varepsilon$ if
$f\in L^{2}(\Omega)$.

Furthermore, if we set $\psi=-\triangle u$, then problem (\ref{PDE1}) is written as
\begin{equation*}
-\varepsilon^{2}\triangle\psi+\psi=f.
\end{equation*}
Note that $\psi$ has boundary layer, but $u$ usually does not have one. Thus,
\begin{equation*}
(\varepsilon^{2}\|\triangle u\|_{L^{2}(\Omega)}^{2}+\|\nabla
u\|_{L^{2}(\Omega)}^{2})^{1/2}=(\varepsilon^{2}\|\psi\|_{L^{2}(\Omega)}^{2}+|u|_{H^{1}(\Omega)}^{2})^{1/2}
\end{equation*}
approaches to $|u|_{H^{1}(\Omega)}$ as
$\varepsilon\rightarrow0^{+}$, which fails to describe the layer of
$\psi$.

An observation of the two decoupled equations
$-\varepsilon^{2}\triangle\psi+\psi=f$ and $-\triangle u=\psi$
suggests that the two measures
$(\varepsilon^{2}\|\nabla\psi\|_{L^{2}(\Omega)}^{2}+\|\psi\|_{L^{2}(\Omega)}^{2}+|u|_{H^{1}(\Omega)}^{2})^{1/2}$
and
$(\varepsilon^{2}\|\nabla\psi\|_{L^{2}(\Omega)}^{2}+\|\psi\|_{L^{2}(\Omega)}^{2})^{1/2}$
can portray the layer of $\psi$ and the first and second derivatives
of $u$. From (\ref{maysp}), we have
\begin{equation*}
\varepsilon\|\nabla \psi\|_{L^{2}(\Omega)}\leq
C\varepsilon\|u\|_{H^{3}(\Omega)}\leq
C\varepsilon^{-1/2}\|f\|_{L^{2}(\Omega)}.
\end{equation*}
Notice that
\begin{equation*}
\|\psi\|_{L^{2}(\Omega)}\leq C\|u\|_{H^{2}(\Omega)}\leq
C\varepsilon^{-1/2}\|f\|_{L^{2}(\Omega)}.
\end{equation*}
If $f\in L^{2}(\Omega)$, then
$\varepsilon\|\nabla\psi\|_{L^{2}(\Omega)}$ and
$\|\psi\|_{L^{2}(\Omega)}$ are balanced with respect to
$\varepsilon$ for the boundary condition (\ref{Boundary
condition2}). These inspire us to think about the mixed finite
element method for the problem (\ref{PDE1}) and the two
aforementioned measures.

However, the mixed finite element method for the problem
(\ref{PDE1}) is a much less explored topic, since there exist some
special problems such as the fourth order problem, where attempts at
using the results of Brezzi and Babu\v{s}ka were not entirely
successful since not all of the stability conditions were satisfied,
cf. \cite{BABUSKA} and the reference therein. To overcome this
difficulty, Falk et al. developed abstract results from which
optimal error estimates for these (biharmonic equation) and other
problems could be derived (\cite{FALK,BREZZI}). However, it is not
easy to extend the results of \cite{FALK} to the problem
(\ref{PDE1}), because of the existence of an extra term and the
singular perturbation parameter $\varepsilon$. Recently, for a
fourth order reaction diffusion equation, the error estimates of its
mixed finite element method was derived in \cite{DANUMJAYA}. We
refer to \cite{CH,GU} about the {\em a posteriori} estimation of
Ciarlet-Raviart methods for the biharmonic equation.

In this work, our goal is to develop robust residual-type {\it a
posteriori estimators for a} mixed finite element method for the
problem (\ref{PDE1}) in the two aforementioned measures. The main
difficulty lies in the fact that the boundary condition
(\ref{Boundary condition2}) does not include any information on the
immediate variable $\psi$. In order to overcome this difficulty, we
develop a novel technique to analyze residual-based {\it a
posteriori} error estimator. The key idea is to replace a function
$v\in H^1(\Omega)$ (such that $-\varepsilon^{2}\triangle v+v\neq0$)
without boundary restriction by a function $\tilde{v}\in
H^1_0(\Omega)$ with boundary restriction, which catches at least
``$\gamma$ times" of $v$ in the $\varepsilon$-weighted energy norm
(see Lemma 3.3 below). Combining this novel design with standard
tools, we develop uniformly robust residual-type {\it a posteriori}
estimators with respect to the singularly perturbed parameter
$\varepsilon$ in the two aforementioned measures. We refer to the
reference \cite{Lin} on balanced norm for mixed formulation for
singularly perturbed reaction-diffusion problems.

The rest of this paper is organized as follows: In Section 2, we
introduce mixed weak formulations and some notations, and prove an
equivalent relation between the primal weak solution and the weak
solution determined by its mixed formulation. Some preliminary
results are provided in Section 3. Residual-type {\it a posteriori}
estimators are developed and proven to be reliable in Section 4. An
efficient lower bound is proved in Section 5. In
Section 6, numerical tests are provided to support our theory.

\section{The mixed weak formulations}
Setting $\psi=-\triangle u$, and employing the boundary condition
(\ref{Boundar condition1}), we attain the Ciarlet-Raviart mixed
problem $P_{1}$:
\begin{equation}\label{mixedproblem1}
 \left \{ \begin{array}{ll}
  -\varepsilon^{2}\triangle\psi+\psi=f\ \ \  & \mbox{in}\ \ \Omega\\
 \hspace{8.0mm}\ \ -\triangle u=\psi & \mbox{in}\ \Omega\\
 \hspace{8.0mm}\ u=\psi=0 & \mbox{on}\ \Gamma.
 \end{array}\right.
\end{equation}
Similarly, using the boundary condition (\ref{Boundary condition2}),
we arrive at the Ciarlet-Raviart mixed formulation $P_{2}$:
\begin{equation}\label{mixedproblem2}
 \left \{ \begin{array}{ll}
  -\varepsilon^{2}\triangle\psi+\psi=f\ \ \  & \mbox{in}\ \ \Omega\\
 \hspace{8.0mm}\ \ -\triangle u=\psi & \mbox{in}\ \Omega\\
 \hspace{5.2mm}\ u=\displaystyle\frac{\partial u}{\partial{\bf n}}=0 & \mbox{on}\ \Gamma.
 \end{array}\right.
\end{equation}

For any bounded open subset $\omega$ of $\Omega$ with Lipschitz
boundary $\gamma$, let $L^{2}(\gamma)$ and $H^{m}(\omega)$ be the standard Lebesgue and Sobolev spaces equipped with standard norms
$\|\cdot\|_{\gamma}=\|\cdot\|_{L^{2}(\gamma)}$ and $\|\cdot\|_{m,\omega}=\|\cdot\|_{H^{m}(\omega)}$,
$m\in\mathbb{N}$ (see \cite{ADAMS} for details). Note that $H^{0}(\omega)=L^{2}(\omega)$. We denote $|\cdot|_{m,\omega}$ the
semi-norm in $H^{m}(\omega)$. Similarly, denote $(\cdot,\cdot)_{\gamma}$ and
$(\cdot,\cdot)_{\omega}$ the $L^{2}$ inner products on $\gamma$
and $\omega$, respectively. We shall omit the symbol $\Omega$ in the notations above if
$\omega=\Omega$.

The weak formulation of problem $P_{1}$ reads: find $(\psi,u)\in
H_{0}^{1}(\Omega)\times H_{0}^{1}(\Omega)$ such that
\begin{equation}\label{discrete mixed formulation P1}
\left \{\begin{array}{ll}
(\varepsilon^{2}\nabla\psi,\nabla\varphi)+(\psi,\varphi)=(f,\varphi)\ \ \  & \forall\ \varphi\in H_{0}^{1}(\Omega),\\
(\nabla u,\nabla v)=(\psi,v)\ \ \ \ \ \ \ \ & \forall\ v\in
H_{0}^{1}(\Omega).
 \end{array}\right.
\end{equation}
The weak formulation of the problem $P_{2}$ reads: find
$(\psi,u)\in H^{1}(\Omega)\times H_{0}^{1}(\Omega)$ such that
\begin{equation}\label{discrete mixed formulation P2}
\left \{\begin{array}{ll}
(\varepsilon^{2}\nabla\psi,\nabla\varphi)+(\psi,\varphi)=(f,\varphi)\ \ \  & \forall\ \varphi\in H_{0}^{1}(\Omega),\\
(\nabla u,\nabla v)=(\psi,v)\ \ \ \ \ \ \ \ & \forall\ v\in
H^{1}(\Omega).
 \end{array}\right.
\end{equation}

Note that, by the Lax-Milgram lemma, both systems (\ref{discrete
mixed formulation P1}) and (\ref{discrete mixed formulation P2})
have a unique solution. In fact, by regularity theory for elliptic
problems \cite{GRISVARD}, if $\Omega$ is convex and $f\in
H^{-1}(\Omega)$, then $u\in H^{3}(\Omega)$ and $\psi\in
H_{0}^{1}(\Omega)$. Thus (\ref{discrete mixed formulation P1}) has
solution, which is unique since its homogeneous system has only one
solution satisfying $(\psi,u)=0$. Similar conclusion can be drawn
for the system (\ref{discrete mixed formulation P2}).

It is well known that the primal weak formulation of
(\ref{PDE1})-(\ref{Boundar condition1}) is: find $\tilde{u}\in
H^{2}(\Omega)\cap H_{0}^{1}(\Omega)$ such that
\begin{equation}\label{may1}
(\varepsilon^{2}\triangle\tilde{u},\triangle
v)+(\nabla\tilde{u},\nabla v)=(f,v),\ \ \ \forall v\in
H^{2}(\Omega)\cap H_{0}^{1}(\Omega),
\end{equation}
and that the one of (\ref{PDE1})-(\ref{Boundary condition2}) is:
find $\tilde{u}\in H_{0}^{2}(\Omega)$ such that
\begin{equation}\label{may2}
(\varepsilon^{2}\triangle\tilde{u},\triangle
v)+(\nabla\tilde{u},\nabla v)=(f,v),\ \ \ \forall v\in
H_{0}^{2}(\Omega).
\end{equation}
The classical results of PDEs imply that (\ref{may1}) and
(\ref{may2}) have unique solutions (see \cite{GRISVARD}). A natural 
question is whether the $u$ determined by (\ref{discrete mixed
formulation P1}) (or (\ref{discrete mixed formulation P2})) is the
solution of (\ref{may1}) (or (\ref{may2})). In \cite{Zhang}, for
biharmonic equation on a reentrant corners polygon, a counterexample
is shown. The following theorems answer this question.

\begin{theorem}\label{mayy}
The solution $\tilde{u}$ of (\ref{may1}) and the $u$ determined by
(\ref{discrete mixed formulation P1}) are identical if and only if $u\in
H^{2}(\Omega)$.
\end{theorem}
\begin{proof}
The necessity is trivial. If the solution of (\ref{discrete mixed
formulation P1}) is such that $u\in H^{2}(\Omega)\cap
H_{0}^{1}(\Omega)$, then we have from the second equation
\begin{equation*}
(-\triangle u,w)=(\psi,w),\ \ \ \forall w\in H_{0}^{1}(\Omega).
\end{equation*}
Notice that $H_{0}^{1}(\Omega)$ is dense in $L^{2}(\Omega)$. It follows that
\begin{equation*}
(-\triangle u,w)=(\psi,w),\ \ \ \forall w\in L^{2}(\Omega).
\end{equation*}

Integration by parts yields
\begin{equation*}
(-\triangle
u,\triangle\varphi)=(\psi,\triangle\varphi)=-(\nabla\psi,\nabla\varphi),\
\ \ \forall\varphi\in H^{2}(\Omega)\cap H_{0}^{1}(\Omega),
\end{equation*}
which implies
\begin{equation}\label{may3}
(\varepsilon^{2}\triangle
u,\triangle\varphi)=(\varepsilon^{2}\nabla\psi,\nabla\varphi),\ \ \
\forall\varphi\in H^{2}(\Omega)\cap H_{0}^{1}(\Omega).
\end{equation}
We obtain from (\ref{may3}) and the second equation of
(\ref{discrete mixed formulation P1}) that
\begin{equation*}
(\varepsilon^{2}\triangle u,\triangle\varphi)+(\nabla
u,\nabla\varphi)=(\varepsilon^{2}\nabla\psi,\nabla\varphi)+(\psi,\varphi),\
\ \ \forall\varphi\in H^{2}(\Omega)\cap H_{0}^{1}(\Omega).
\end{equation*}
In terms of (\ref{may1}), we proved that $u$ is the solution of
(\ref{may1}).
\end{proof}

\begin{theorem}
The solution $\tilde{u}$ of (\ref{may2}) and the $u$ determined by
(\ref{discrete mixed formulation P2}) are identical if and only if $u\in
H_{0}^{2}(\Omega)$.
\end{theorem}
\begin{proof}
The necessity is trivial. From the second equation of (\ref{discrete
mixed formulation P2}), integration by parts, and variational
principle, we know that the Neumann boundary condition $\partial
u/\partial{\bf n}=0$ on $\Gamma$ is automatically satisfied.
Following the proof of Theorem \ref{mayy}, we know that if the
solution $u$ of (\ref{discrete mixed formulation P2}) is in
$H_{0}^{2}(\Omega)$, then $u$ is the solution of (\ref{may2}).
\end{proof}

Let $\mathcal{T}_{h}$ be a shape regular partition of $\Omega$ into
 triangles (tetrahedra for $d=3$) or parallelograms (parallelepiped
for $d=3$) satisfying the angle condition \cite{Ciarlet}, i.e.,
there exists a constant $C_{0}$ such that
\begin{equation}\label{PDE4}
C_{0}^{-1}h_{K}^{d}\leq|K|\leq C_{0}h_{K}^{d}\ \ \ \ \ \ \forall \
K\in\mathcal{T}_{h},
\end{equation}
where $h_{K} :={\rm diam}(K)$. Let $P_{k}(K)$ be the space of {
polynomials of total degree at most $k$ if $K$ is a simplex, or the
space of polynomials with degree at most $k$ for each variable if
$K$ is a parallelogram/parallelepiped}. Define the finite
element spaces $V_{h}$ and $V_{h}^{0}$ by
\begin{equation*}
V_{h} :=\{v_{h}\in C(\overline{\Omega}) :v_{h}|_{K}\in P_{k}(K), \
 \forall K\in\mathcal{T}_{h}\}
\end{equation*}
and
\begin{equation*}
V_{h}^{0} :=\{v_{h}\in V_{h} :v_{h}|_{\Gamma}=0\},
\end{equation*}
respectively.

We introduce the mixed finite element method for problem $P_{1}$:
find $(\psi_{h},u_{h})\in V_{h}^{0}\times V_{h}^{0}$ such that
\begin{equation}\label{mixed from P1}
\left \{\begin{array}{ll}
(\varepsilon^{2}\nabla\psi_{h},\nabla\varphi_{h})+(\psi_{h},\varphi_{h})=(f,\varphi_{h})\ \ \  & \forall\ \varphi_{h}\in V_{h}^{0},\\
(\nabla u_{h},\nabla v_{h})=(\psi_{h},v_{h})\ \ \ \ \ \ \ \ & \forall\ v_{h}\in V_{h}^{0}.
 \end{array}\right.
\end{equation}
For problem $P_{2}$, the mixed problem reads: find $(\psi_{h},u_{h})\in V_{h}\times
V_{h}^{0}$ such that
\begin{equation}\label{mixed form P2}
\left \{\begin{array}{ll}
(\varepsilon^{2}\nabla\psi_{h},\nabla\varphi_{h})+(\psi_{h},\varphi_{h})=(f,\varphi_{h})\ \ \  & \forall\ \varphi_{h}\in V_{h}^{0},\\
(\nabla u_{h},\nabla v_{h})=(\psi_{h},v_{h})\ \ \ \ \ \ \ \ &
\forall\ v_{h}\in V_{h}.
\end{array}\right.
\end{equation}

By standard arguments, problem (\ref{mixed from P1}) possesses a
unique solution provided there exist functions $\psi_{h}$ and
$u_{h}$ satisfying
\begin{equation}\label{mixed form P11}
\left \{\begin{array}{ll}
(\varepsilon^{2}\nabla\psi_{h},\nabla\varphi_{h})+(\psi_{h},\varphi_{h})=0\ \ \  & \forall\ \varphi_{h}\in V_{h}^{0},\\
(\nabla u_{h},\nabla v_{h})=(\psi_{h},v_{h})\ \ \ \ \ \ \ \ & \forall\ v_{h}\in V_{h}^{0},
 \end{array}\right.
\end{equation}
then $(\psi_{h},u_{h})$ is the trivial solution to the system. In fact, taking $\varphi_{h}=\psi_{h}$ in
the first equation of (\ref{mixed form P11}), one gets $\psi_{h}=0$.
Setting $v_{h}=u_{h}$ in the second equation of (\ref{mixed form
P11}), one obtains $u_{h}=0$. Similarly, it is verified that problem (\ref{mixed form P2}) has also a unique solution.

We define a measure of the error between the exact solution
$(\psi,u)$ and the numerical solution $(\psi_{h},u_{h})$ by
\begin{equation*}
\|(\psi-\psi_{h},u-u_{h})\|^{2}
:=\|\psi-\psi_{h}\|_{\mathcal{E}}^{2}+|u-u_{h}|_{1}^{2},
\end{equation*}
where
\begin{equation*}
\|\psi-\psi_{h}\|_{\mathcal{E}}=\big(\varepsilon^{2}|\psi-\psi_{h}|_{1}^{2}+\|\psi-\psi_{h}\|^{2}\big)^{1/2}
\end{equation*}
is the standard energy norm of the numerical error $\psi-\psi_{h}$.
In this paper, we aim at robust {\it a posterior} error estimators
for the numerical errors $\|\psi-\psi_{h}\|_{\mathcal{E}}$,
$|u-u_{h}|_{1}$, and $\|(\psi-\psi_{h},u-u_{h})\|$.

We next introduce some notations that will be used later. We
denote $\mathcal{E}_{h}^{0}$ the set of interior sides (if $d=2$) or faces (if $d=3$) in $\mathcal{T}_{h}$,
$\mathcal{E}_{T}$ the set of sides or faces of $T\in\mathcal{T}_{h}$, and
$\tilde{\omega}_{T}$ the union of all elements in
$\mathcal{T}_{h}$ sharing at least one point with $T$.
 For a side or face $E$ in $\mathcal{E}_{h}$, which is the set of element
 sides or faces in $\mathcal{T}_{h}$, let $h_{E}$ be the diameter of $E$,
 and $\omega_{E}$ be the union of all elements in $\mathcal{T}_{h}$ sharing $E$.
For a function $v$ in the ``broken Sobolev space" $H^{1}(\bigcup\mathcal{T}_{h})$, we define $[v]|_{E} :=(v|_{T_{+}})|_{E}-(v|_{T_{-}})|_{E}$ as
the jump of $v$ across an interior side or face $E$, where
$T_{+}$ and $T_{-}$ are the two neighboring elements such that $E =T_{+}\cap T_{-}$.

Throughout of this paper, we denote by $C_{Q}$ a constant
depending only on $Q$, and denote by $C_{i}\ (i=0, 1,
\cdot\cdot\cdot)$ constants depending on the mesh shape
regularity and $d$. In what follows we use the notation $A\lesssim
F$ to represent $A\leq CF$ with a generic constant $C>0$
independent of mesh size. In addition, $A\approx F$ abbreviates
$A\lesssim F\lesssim A$.

\section{Preliminary results}
For problem $P_{2}$, $\psi$ and $u$ are decoupled. However, $\psi$
does not obtain any information directly from boundary conditions.
It will be difficult to develop residual-based {\it a posteriori}
error estimates if the residual on the boundary is not clear. To
overcome this difficulty, we shall develop a novel analytical
technique (see Section 4), which is based on the following
approximation result.
\begin{lemma}\label{Julylemma1}
Let $v\in H^{1}(\Omega)/H_{0}^{1}(\Omega)$ which satisfies
$-\varepsilon^{2}\triangle v+v=0$ (operator to be understood in weak
sense). Then it holds
\begin{equation*}
\displaystyle\inf_{w\in
H_{0}^{1}(\Omega)}\frac{\varepsilon^{2}|v-w|_{1}^{2}+||v-w||^{2}}{\varepsilon^{2}|v|_{1}^{2}+||v||^{2}}=1.
\end{equation*}
\end{lemma}
\begin{proof}
Consider the functional
\begin{equation*}
J(w)=\varepsilon^{2}|v-w|_{1}^{2}+||v-w||^{2},\ \ \forall\ w\in
H_{0}^{1}(\Omega).
\end{equation*}
Minimization of such functional in $H_{0}^{1}(\Omega)$ immediately
leads to the following variational problem: Find $\tilde{v}\in
H_{0}^{1}(\Omega)$ such that
\begin{equation}\label{July1}
\displaystyle\int_{\Omega}\big(\varepsilon^{2}\nabla(v-\tilde{v})\cdot\nabla\phi+(v-\tilde{v})\phi\big)d{\bf
x}=0,\ \ \forall\ \phi\in H_{0}^{1}(\Omega).
\end{equation}
Integrating by parts, we arrive at
\begin{equation*}
\displaystyle\int_{\Omega}\big(-\varepsilon^{2}\triangle(v-\tilde{v})+v-\tilde{v}\big)\phi
d{\bf x}=0,\ \ \forall\ \phi\in H_{0}^{1}(\Omega),
\end{equation*}
which implies
\begin{equation*}
-\varepsilon^{2}\triangle\tilde{v}+\tilde{v}=-\varepsilon^{2}\triangle
v+v=0.
\end{equation*}
So $\tilde{v}$ is the solution to the following problem:
\begin{equation}\label{July2}
    \left\{
    \begin{array}{ll}
        -\varepsilon^{2}\triangle\tilde{v}+\tilde{v}=0 & \mbox{in}~\Omega,\\
        \tilde{v}=0 & \mbox{on}~\Gamma.
     \end{array}
     \right.
\end{equation}

Since the problem (\ref{July2}) has only trivial solution, we have
\begin{align*}
    & \inf_{w\in H_{0}^{1}(\Omega)}\frac{\varepsilon^{2}|v-w|_{1}^{2}+||v-w||^{2}}{\varepsilon^{2}|v|_{1}^{2}+||v||^{2}} = \frac{\inf_{w\in H_{0}^{1}(\omega)}J(w)}{\varepsilon^{2}|v|_{1}^{2}+||v||^{2}} \\
    = & \frac{\varepsilon^{2}|v-\tilde{v}|_{1}^{2} + ||v-\tilde{v}||^{2}}{\varepsilon^{2}|v|_{1}^{2}+||v||^{2}} = \frac{\varepsilon^{2}|v|_{1}^{2}+||v||^{2}}{\varepsilon^{2}|v|_{1}^{2}+||v||^{2}} = 1,
\end{align*}
which completes the proof.
\end{proof}

\begin{lemma}\label{Julylemma2}
If there holds the following relation
\begin{equation*}
\displaystyle\inf_{w\in
H_{0}^{1}(\Omega)}\frac{\varepsilon^{2}|v-w|_{1}^{2}+||v-w||^{2}}{\varepsilon^{2}|v|_{1}^{2}+||v||^{2}}=1,
\end{equation*}
then $v$ satisfies $-\varepsilon^{2}\triangle v+v=0$.
\end{lemma}
\begin{proof}
Since $\tilde{v}=0\in H_{0}^{1}(\Omega)$, the condition
\begin{equation*}
1=\displaystyle\inf_{w\in
H_{0}^{1}(\Omega)}\frac{\varepsilon^{2}|v-w|_{1}^{2}+||v-w||^{2}}{\varepsilon^{2}|v|_{1}^{2}+||v||^{2}}
\end{equation*}
can be satisfied when $w=\tilde{v}$. On the other hand, $\tilde{v}$
is the solution to the variational problem: Find $\tilde{v}\in
H_{0}^{1}(\Omega)$ such that
\begin{equation*}
\displaystyle\int_{\Omega}\big(\varepsilon^{2}\nabla(v-\tilde{v})\cdot\nabla\phi+(v-\tilde{v})\phi\big)d{\bf
x}=0,\ \ \forall\ \phi\in H_{0}^{1}(\Omega).
\end{equation*}
By integrating by parts, we obtain
\begin{equation*}
\displaystyle\int_{\Omega}\big(-\varepsilon^{2}\triangle(v-\tilde{v})+(v-\tilde{v})\big)\phi
d{\bf x}=0,\ \ \forall\ \phi\in H_{0}^{1}(\Omega),
\end{equation*}
which leads to
$-\varepsilon^{2}\triangle(v-\tilde{v})+(v-\tilde{v})=0$, this means
\begin{equation*}
-\varepsilon^{2}\triangle
v+v=-\varepsilon^{2}\triangle\tilde{v}+\tilde{v}=0.
\end{equation*}
\end{proof}

\begin{lemma}\label{Julylemma3}
Let $v\in H^{1}(\Omega)$ such that $-\varepsilon^{2}\triangle
v+v\neq0$. Then there exists $\gamma\in(0,1)$ such that
\begin{equation*}
\displaystyle\inf_{w\in
H_{0}^{1}(\Omega)}\frac{\varepsilon^{2}|v-w|_{1}^{2}+||v-w||^{2}}{\varepsilon^{2}|v|_{1}^{2}+||v||^{2}}\leq\gamma.
\end{equation*}
\end{lemma}
\begin{proof}
We only prove the case $v\in H^{1}(\Omega)/H_{0}^{1}(\Omega)$ using
proof by contradiction, since $v\in H_{0}^{1}(\Omega)$ is obvious.
Assume that there does not exist $\gamma\in(0,1)$ such that
\begin{equation*}
\displaystyle\inf_{w\in
H_{0}^{1}(\omega)}(\varepsilon^{2}|v-w|_{1}^{2}+||v-w||^{2})\leq\gamma(\varepsilon^{2}|v|_{1}^{2}+||v||^{2}),
\end{equation*}
which means
\begin{equation*}
\displaystyle\inf_{w\in
H_{0}^{1}(\omega)}(\varepsilon^{2}|v-w|_{1}^{2}+||v-w||^{2})>\gamma(\varepsilon^{2}|v|_{1}^{2}+||v||^{2})
\end{equation*}
for all $\gamma\in(0,1)$. From the proof of Lemma \ref{Julylemma1}, there exists $\tilde{v}\in H_{0}^{1}(\Omega)$ such that
\begin{equation*}
\begin{array}{lll}
\varepsilon^{2}|v|_{1}^{2}+||v||^{2}&\geq&\displaystyle\inf_{w\in
H_{0}^{1}(\Omega)}(\varepsilon^{2}|v-w|_{1}^{2}+||v-w||^{2})\vspace{2mm}\\
&=&\varepsilon^{2}|v-\tilde{v}|_{1}^{2}+||v-\tilde{v}||^{2}\vspace{2mm}\\
&>&\gamma(\varepsilon^{2}|v|_{1}^{2}+||v||^{2}).
\end{array}
\end{equation*}
In particular, for any $n\in \mathbb{N}$, let $\gamma=\frac{n-1}{n}$, it holds that
\begin{equation*}
\varepsilon^{2}|v|_{1}^{2}+||v||^{2}\geq\varepsilon^{2}|v-\tilde{v}|_{1}^{2}+||v-\tilde{v}||^{2}>
\frac{n-1}{n}(\varepsilon^{2}|v|_{1}^{2}+||v||^{2}).
\end{equation*}
Since $\varepsilon^{2}|v-\tilde{v}|_{1}^{2}+||v-\tilde{v}||^{2}$ is
the upper bound for
$\frac{n-1}{n}(\varepsilon^{2}|v|_{1}^{2}+||v||^{2})$ with respect
to the positive integer number $n$, and
$\varepsilon^{2}|v|_{1}^{2}+||v||^{2}$ is its supremum. Therefore,
\begin{equation*}
\varepsilon^{2}|v|_{1}^{2}+||v||^{2}\leq\varepsilon^{2}|v-\tilde{v}|_{1}^{2}+||v-\tilde{v}||^{2}\leq
\varepsilon^{2}|v|_{1}^{2}+||v||^{2},
\end{equation*}
which means
\begin{equation*}
\varepsilon^{2}|v-\tilde{v}|_{1}^{2}+||v-\tilde{v}||^{2}=
\varepsilon^{2}|v|_{1}^{2}+||v||^{2}.
\end{equation*}
This leads to
\begin{equation*}
\displaystyle\inf_{w\in
H_{0}^{1}(\Omega)}\frac{\varepsilon^{2}|v-w|_{1}^{2}+||v-w||^{2}}{\varepsilon^{2}|v|_{1}^{2}+||v||^{2}}
=\displaystyle\frac{\varepsilon^{2}|v-\tilde{v}|_{1}^{2}+
||v-\tilde{v}||^{2}}{\varepsilon^{2}|v|_{1}^{2}+||v||^{2}}=1.
\end{equation*}
From Lemma \ref{Julylemma2}, we have $-\varepsilon^{2}\triangle
v+v=0$, this leads to a contradiction. We complete the proof.
\end{proof}

\begin{remark}
Note that we can not prove that $\big\{v\in H^{1}(\Omega):
-\triangle v+v\neq0\big\}$ is dense in $H_{0}^{1}(\Omega)$ by
recursion by using Lemma \ref{Julylemma3}, because of
$-\triangle(v-\tilde{v})+(v-\tilde{v})=0$.
\end{remark}

Denote by $I_{h} :L^{2}(\Omega)\rightarrow V_{h}^{0}$ the
quasi-interpolation operator of Cl\'{e}ment (cf.
\cite{Ciarlet,Clement,Verfurth2}).
\begin{lemma}\label{year0}
For all $T\in\mathcal{T}_{h},E\subset\partial T$, define
$\alpha_{T}$ and $\alpha_{E}$ the weighted factors by
\begin{equation*}
\alpha_{T} :=\min\{h_{T}\varepsilon^{-1},1\}\ \ {\rm and}\ \
\alpha_{E} :=\varepsilon^{-1/2}\min\{h_{T}\varepsilon^{-1},1\},
\end{equation*}
respectively. Then the following local error estimates hold for
$v\in H^{1}(\tilde{\omega}_{T})$:
\begin{equation}\label{year1}
\|v-I_{h}v\|_{T}\lesssim\alpha_{T}\|v\|_{\mathcal{E},\tilde{\omega}_{T}}
\end{equation}
and
\begin{equation}\label{year2}
\|v-I_{h}v\|_{E}\lesssim\alpha_{E}\|v\|_{\mathcal{E},\tilde{\omega}_{T}}.
\end{equation}
\end{lemma}
\begin{proof}
Following the line of the proof of Lemma 3.2 in \cite{Verfurth2}, we
obtain the desired estimates (\ref{year1}) and (\ref{year2}).
\end{proof}

For $\theta\in(0,1], T\in\mathcal{T}_{h},E\in\mathcal{E}_{h}$,
denote $\psi_{T}$ and $\psi_{E,\theta}$ the two bubble functions
defined in \cite{Verfurth2}, and $P_{E}$ a continuation operator
introduced in \cite{Verfurth2} by
\begin{equation*}
P_{E} :L^{\infty}(E)\rightarrow L^{\infty}(\omega_{E}),
\end{equation*}
which maps polynomials onto piecewise polynomials of the same
degree.
\begin{lemma}\label{year3}
The following estimates hold for all $v\in\mathcal{P}_{k}$ (the set
of polynomials of degree at most $k$) and $T\in\mathcal{T}_{h}$
\begin{equation}\label{year4}
\|v\|_{T}^{2}\lesssim (v,\psi_{T}v)_{T},
\end{equation}
\begin{equation}\label{year5}
\|v\psi_{T}\|_{T}\leq\|v\|_{T},
\end{equation}
\begin{equation}\label{year6}
\|v\psi_{T}\|_{\mathcal{E},T}\lesssim\alpha_{T}^{-1}\|v\|_{T}.
\end{equation}
Furthermore, for $E\in\mathcal{E}_{h}$, set $\theta_{E}
:=\min\{\varepsilon h_{E}^{-1},1\}$. Then there hold the following
estimates for all $E\in\mathcal{E}_{h}$ and
$\sigma\in\mathcal{P}_{k}|_{E}$.
\begin{equation}\label{year7}
\|\sigma\|_{E}^{2}\lesssim(\sigma,\psi_{E,\theta_{E}}P_{E}\sigma)_{E},
\end{equation}
\begin{equation}\label{year8}
\|\psi_{E,\theta_{E}}P_{E}\sigma\|_{\omega_{E}}\lesssim\varepsilon^{1/2}\min\{h_{E}\varepsilon^{-1},1\}^{1/2}\|\sigma\|_{E},
\end{equation}
\begin{equation}\label{year9}
\|\psi_{E,\theta_{E}}P_{E}\sigma\|_{\mathcal{E},\omega_{E}}\lesssim\varepsilon^{1/2}\min\{h_{E}\varepsilon^{-1},1\}^{-1/2}\|\sigma\|_{E}.
\end{equation}
\end{lemma}
\begin{proof}
Following the line of the proof of Lemma 3.3 in \cite{Verfurth2}, we
attain (\ref{year4})-(\ref{year9}).
\end{proof}

\section{A reliable upper bound}
For all $T\in\mathcal{T}_{h}$, define $\eta_{\psi,T}$ and
$\eta_{u,T}$ the elementwise indicators of $\psi$ and $u$, respectively,
by
\begin{equation*}
    \eta_{\psi,T}
    := \Big\{ \alpha_{T}^{2} \|f+\varepsilon^{2}\triangle\psi_{h}-\psi_{h}\|_{T}^{2} + \frac{1}{2} \sum_{E\in\mathcal{E}_{T}\cap\mathcal{E}_{h}^{0}} \alpha_{E}^{2} \Big\| [\varepsilon^{2}\frac{\partial\psi_{h}}{\partial{\bf n}}] \Big\|_{E}^{2} \Big\}^{1/2}
\end{equation*}
and
\begin{equation*}
    \eta_{u,T} := \Big\{h_{T}^{2}\|\triangle u_{h}+\psi_{h}\|_{T}^{2} + \frac{1}{2} \sum_{E\in\mathcal{E}_{T}\cap\mathcal{E}_{h}^{0}}h_{E} \Big\|[\frac{\partial u_{h}}{\partial {\bf n}}] \Big\|_{E}^{2} \Big\}^{1/2}.
\end{equation*}

\begin{theorem}
Let $(\psi,u)\in H_{0}^{1}(\Omega)\times H_{0}^{1}(\Omega)$ and
$(\psi_{h},u_{h})\in V_{h}^{0}\times V_{h}^{0}$ be the solutions to
(\ref{discrete mixed formulation P1}) and (\ref{mixed from P1}),
respectively. Then there exist positive constants $C_{1},C_{2}$, and
$C_{3}$, independent of the mesh-size function $h$ and
$\varepsilon$, such that
\begin{equation}\label{year10}
    \|\psi-\psi_{h}\|_{\mathcal{E}}\leq C_{1} \Big\{\sum\limits_{T\in\mathcal{T}_{h}}\eta_{\psi,T}^{2}\Big\}^{1/2},
\end{equation}
\begin{equation}\label{year11}
    |u-u_{h}|_{1}\leq C_{2} \Big\{\sum\limits_{T\in\mathcal{T}_{h}}\eta_{\psi,T}^{2}+\eta_{u,T}^{2}\Big\}^{1/2},
\end{equation}
\begin{equation}\label{year12}
    \|(\psi-\psi_{h},u-u_{h})\|\leq C_{3} \Big\{\sum\limits_{T\in\mathcal{T}_{h}}\eta_{\psi,T}^{2}+\eta_{u,T}^{2}\Big\}^{1/2}.
\end{equation}
\end{theorem}
\begin{proof}
From the definition of the measure $\|(\psi-\psi_{h},u-u_{h})\|$,
(\ref{year12}) follows from (\ref{year10}) and (\ref{year11}). We
need to prove (\ref{year10}) and (\ref{year11}). We have from the
first equations of (\ref{discrete mixed formulation P1}) and
(\ref{mixed from P1}) that
\begin{equation}\label{year16}
(\varepsilon^{2}\nabla(\psi-\psi_{h}),\nabla\varphi_{h})+(\psi-\psi_{h},\varphi_{h})=0,\
\ \ \forall\ \varphi_{h}\in V_{h}^{0}.
\end{equation}
For any $\varphi\in H_{0}^{1}(\Omega)$, let $\varphi_{h}$ be the
Clem\'{e}nt interpolation of $\varphi$ in $V_{h}^{0}$, i.e.,
$\varphi_{h}=I_{h}\varphi$. Applying integration by parts and
(\ref{year16}), we get
\begin{align}
    \nonumber
    & (\varepsilon^{2}\nabla(\psi-\psi_{h}),\nabla\varphi)+(\psi-\psi_{h},\varphi) \\
    \nonumber
    = & (\varepsilon^{2}\nabla(\psi-\psi_{h}),\nabla(\varphi-\varphi_{h}))+(\psi-\psi_{h},\varphi-\varphi_{h}) \\
    \nonumber
    = & \sum\limits_{T\in\mathcal{T}_{h}}\int_{T}(-\varepsilon^{2}\triangle\psi+\varepsilon^{2}\triangle\psi_{h}+\psi-\psi_{h}) (\varphi-\varphi_{h}) + \int_{\partial T}\varepsilon^{2}\frac{\partial(\psi-\psi_{h})}{\partial{\bf n}}(\varphi-\varphi_{h}) \\
    \nonumber
    = & \sum\limits_{T\in\mathcal{T}_{h}}\int_{T}(f+\varepsilon^{2}\triangle\psi_{h}-\psi_{h}) (\varphi-\varphi_{h})-\int_{\partial T} \varepsilon^{2} \frac{\partial\psi_{h}}{\partial{\bf n}} (\varphi-\varphi_{h})\\
    \label{year17}
    \leq & \sum\limits_{T\in\mathcal{T}_{h}} \Big\{\|f+\varepsilon^{2}\triangle\psi_{h}-\psi_{h}\|_{T} \|\varphi-\varphi_{h}\|_{T} + \frac{1}{2} \sum\limits_{E\in\mathcal{E}_{T}\cap\mathcal{E}_{h}^{0}} \Big\|[\varepsilon^{2}\frac{\partial\psi_{h}}{\partial{\bf n}}]\Big\|_{E} \|\varphi-\varphi_{h}\|_{E} \Big\}.
\end{align}

Notice that
\begin{equation}\label{year18}
\begin{array}{lll}
\|\psi-\psi_{h}\|_{\mathcal{E}}&=&\displaystyle\frac{(\varepsilon^{2}\nabla(\psi-\psi_{h}),\nabla(\psi-\psi_{h}))+
(\psi-\psi_{h},\psi-\psi_{h})}{\|\psi-\psi_{h}\|_{\mathcal{E}}}\vspace{2mm}\\
&\leq&\displaystyle\sup\limits_{0\neq\varphi\in
H_{0}^{1}(\Omega)}\frac{(\varepsilon^{2}\nabla(\psi-\psi_{h}),\nabla\varphi)+(\psi-\psi_{h},\varphi)}{\|\varphi\|_{\mathcal{E}}}.
\end{array}
\end{equation}
The first estimate (\ref{year10}) follows from a combination of
(\ref{year18}), (\ref{year17}), and (\ref{year1})-(\ref{year2}).

We next prove (\ref{year11}). From the second equation of (\ref{discrete mixed formulation P1})
and (\ref{mixed from P1}), we get
\begin{equation}\label{year19}
(\nabla(u-u_{h}),\nabla v_{h})=(\psi-\psi_{h},v_{h}),\ \ \ \forall
v_{h}\in V_{h}^{0}.
\end{equation}
Similarly, we have, for any $v\in H_{0}^{1}(\Omega)$ and $v_{h}=I_{h}v$,
\begin{align}
    \nonumber
    & (\nabla(u-u_{h}),\nabla v)=(\nabla(u-u_{h}),\nabla(v-v_{h}))+(\nabla(u-u_{h}),\nabla v_{h}) \\
    \nonumber
    = & \sum\limits_{T\in\mathcal{T}_{h}}\int_{T}(-\triangle u+\triangle u_{h})(v-v_{h})+\int_{\partial T}\frac{\partial(u-u_{h})}{\partial{\bf n}}(v-v_{h})+(\psi-\psi_{h},v_{h}) \\
    \nonumber
    = & \sum\limits_{T\in\mathcal{T}_{h}}\int_{T}(-\triangle u+\triangle u_{h}-\psi+\psi_{h})(v-v_{h})+\int_{\partial T}\frac{\partial(u-u_{h})}{\partial{\bf n}}(v-v_{h}) + (\psi-\psi_{h},v) \\
    \nonumber
    = & \sum\limits_{T\in\mathcal{T}_{h}}\int_{T}(\triangle u_{h}+\psi_{h})(v-v_{h})-\int_{\partial T}\frac{\partial u_{h}}{\partial{\bf n}}(v-v_{h})+(\psi-\psi_{h},v) \\
    \nonumber
    \leq & \sum\limits_{T\in\mathcal{T}_{h}}\Big\{\|\triangle u_{h}+\psi_{h}\|_{T} \|v-v_{h}\|_{T} + \frac{1}{2} \sum\limits_{E\in\mathcal{E}_{T}\cap\mathcal{E}_{h}^{0}} \Big\|[\frac{\partial u_{h}}{\partial{\bf n}}] \Big\|_{E}\|v-v_{h}\|_{E}\Big\} \\
    \label{year20}
    & + \|\psi-\psi_{h}\|\|v\|.
\end{align}

Recall the following estimates on Clem\'{e}nt interpolation (cf. \cite{Clement}):
\begin{equation}\label{year20+}
\|v-I_{h}v\|_{T}\lesssim h_{T}|v|_{1,\tilde{\omega}_{T}}\ \ \ {\rm
for\ all}\ \ T\in\mathcal{T}_{h},\ \ v\in H^{1}(\tilde{\omega}_{T})
\end{equation}
and
\begin{equation}\label{year20++}
\|v-I_{h}v\|_{E}\lesssim h_{T}^{1/2}|v|_{1,\tilde{\omega}_{T}}\ \ \
{\rm for\ all}\ \ E\in\mathcal{E}_{h}, E\subset\partial T, \ \ v\in
H^{1}(\tilde{\omega}_{T}).
\end{equation}
For any $v\in H_{0}^{1}(\Omega)$, the Poincar\'{e} inequality implies
\begin{equation}\label{year21}
\|v\|\leq\|v\|_{1}\lesssim|v|_{1}.
\end{equation}
A combination of (\ref{year20}) and (\ref{year20+})-(\ref{year21}) yields
\begin{align}
    \nonumber
    & (\nabla(u-u_{h}),\nabla v) \\
    \label{year21+}
    \lesssim & \Big\{\sum\limits_{T\in\mathcal{T}_{h}}\Big(h_{T}^{2}\|\triangle u_{h}+\psi_{h}\|_{T}^{2}+ \frac{1}{2} \sum\limits_{E\in\mathcal{E}_{T}\cap\mathcal{E}_{h}^{0}}h_{T} \Big\|[\frac{\partial u_{h}}{\partial{\bf n}}]\Big\|_{E}^{2} \Big) + \|\psi-\psi_{h}\| \Big\} |v|_{1}.
\end{align}

Notice that
\begin{equation}\label{year22}
|u-u_{h}|_{1}\leq\displaystyle\sup\limits_{0\neq v\in H_{0}^{1}(\Omega)}\frac{(\nabla(u-u_{h}),\nabla v)}{|v|_{1}}.
\end{equation}
A combination of (\ref{year22}), (\ref{year21+}), and (\ref{year10})
yields
\begin{equation*}
|u-u_{h}|_{1}\lesssim\Big\{\displaystyle\sum\limits_{T\in\mathcal{T}_{h}}\eta_{u,T}^{2}+\eta_{\psi,T}^{2}\Big\}^{1/2}.
\end{equation*}
This completes the proof of (\ref{year11}).
\end{proof}

\begin{theorem}
Let $(\psi,u)\in H^{1}(\Omega)\times H_{0}^{1}(\Omega)$ and
$(\psi_{h},u_{h})\in V_{h}\times V_{h}^{0}$ be the solutions to
(\ref{discrete mixed formulation P2}) and (\ref{mixed form P2}),
respectively. If
$-\varepsilon^{2}\triangle(\psi-\psi_{h})+(\psi-\psi_{h})\neq0$,
then there exist positive constants $C_{4},C_{5}$, and $C_{6}$,
independent of the mesh-size function $h$ and $\varepsilon$, such
that
\begin{equation}\label{year13}
\|\psi-\psi_{h}\|_{\mathcal{E}}\leq
C_{4}\displaystyle\Big\{\sum\limits_{T\in\mathcal{T}_{h}}\eta_{\psi,T}^{2}\Big\}^{1/2},
\end{equation}
\begin{equation}\label{year14}
|u-u_{h}|_{1}\leq
C_{5}\displaystyle\Big\{\sum\limits_{T\in\mathcal{T}_{h}}\eta_{\psi,T}^{2}+\eta_{u,T}^{2}\Big\}^{1/2},
\end{equation}
\begin{equation}\label{year15}
\|(\psi-\psi_{h},u-u_{h})\|\leq
C_{6}\displaystyle\Big\{\sum\limits_{T\in\mathcal{T}_{h}}\eta_{\psi,T}^{2}+\eta_{u,T}^{2}\Big\}^{1/2}.
\end{equation}
\end{theorem}
\begin{proof}
We have from
$-\varepsilon^{2}\triangle(\psi-\psi_{h})+(\psi-\psi_{h})\neq0$
\begin{equation}\label{year23}
\|\psi-\psi_{h}\|_{\mathcal{E}}\leq\displaystyle\sup\limits_{v\in
H^{1}(\Omega),-\varepsilon^{2}\triangle
v+v\neq0}\frac{(\varepsilon^{2}\nabla(\psi-\psi_{h}),\nabla
v)+(\psi-\psi_{h},v)}{\|v\|_{\mathcal{E}}}.
\end{equation}
For $v\in H^{1}(\Omega)$ satisfying $-\varepsilon^{2}\triangle
v+v\neq0$, from the proofs of Lemmas \ref{Julylemma1} and
\ref{Julylemma3}, there exist $\tilde{v}\in H_{0}^{1}(\Omega)$ and
$\gamma\in(0,1)$, such that
\begin{equation}\label{year24}
\begin{array}{lll}
\varepsilon^{2}|v-\tilde{v}|_{1}^{2}+||v-\tilde{v}||^{2}&=&\displaystyle\inf_{w\in
H_{0}^{1}(\Omega)}(\varepsilon^{2}|v-w|_{1}^{2}+||v-w||^{2})\vspace{2mm}\\
&\leq&\gamma(\varepsilon^{2}|v|_{1}^{2}+||v||^{2})=\gamma\|v\|_{\mathcal{E}}^{2}.
\end{array}
\end{equation}

Let $\tilde{v}_{h}=I_{h}\tilde{v}$ be the Clem\'{e}nt interpolation
of $\tilde{v}$ in $V_{h}^{0}$. From the first equation of
(\ref{discrete mixed formulation P2}) and (\ref{mixed form P2}), we
have
\begin{equation}\label{year25}
(\varepsilon^{2}\nabla(\psi-\psi_{h}),\nabla\tilde{v}_{h})+(\psi-\psi_{h},\tilde{v}_{h})=0.
\end{equation}

From (\ref{year25}), we have
\begin{equation}\label{year25+}
\begin{array}{lll}
&\ &(\varepsilon^{2}\nabla(\psi-\psi_{h}),\nabla v)+(\psi-\psi_{h},v)=
(\varepsilon^{2}\nabla(\psi-\psi_{h}),\nabla(v-\tilde{v}))+(\psi-\psi_{h},v-\tilde{v})\vspace{2mm}\\
&\ &\ \hspace{54mm}+(\varepsilon^{2}\nabla(\psi-\psi_{h}),\nabla\tilde{v})+(\psi-\psi_{h},\tilde{v})\vspace{2mm}\\
&\ &\
\leq\|\psi-\psi_{h}\|_{\mathcal{E}}\|v-\tilde{v}\|_{\mathcal{E}}+
(\varepsilon^{2}\nabla(\psi-\psi_{h}),\nabla(\tilde{v}-\tilde{v}_{h}))+(\psi-\psi_{h},\tilde{v}-\tilde{v}_{h}).
\end{array}
\end{equation}
Repeating the proof of (\ref{year17}), and applying (\ref{year1})-(\ref{year2}), we have
\begin{equation}\label{year26}
\begin{array}{lll}
&\ &(\varepsilon^{2}\nabla(\psi-\psi_{h}),\nabla(\tilde{v}-\tilde{v}_{h}))+(\psi-\psi_{h},\tilde{v}-\tilde{v}_{h})\vspace{2mm}\\
&\
&\leq\displaystyle\sum\limits_{T\in\mathcal{T}_{h}}\Big\{\|f+\varepsilon^{2}\triangle\psi_{h}-\psi_{h}\|_{T}
\|\tilde{v}-\tilde{v}_{h}\|_{T}+
\frac{1}{2}\sum\limits_{E\in\mathcal{E}_{T}\cap\mathcal{E}_{h}^{0}} \Big\|[\varepsilon^{2}\frac{\partial\psi_{h}}{\partial{\bf
n}}]\Big\|_{E}\|\tilde{v}-\tilde{v}_{h}\|_{E}\Big\}\vspace{2mm}\\
&\ &\leq\displaystyle
C\Big\{\sum\limits_{T\in\mathcal{T}_{h}}\eta_{\psi,
T}^{2}\Big\}^{1/2}\|\tilde{v}\|_{\mathcal{E}}.
\end{array}
\end{equation}
Using the triangle inequality and (\ref{year24}), we have
\begin{equation}\label{year27}
\|\tilde{v}\|_{\mathcal{E}}^{2}\lesssim\|v-\tilde{v}\|_{\mathcal{E}}^{2}+\|v\|_{\mathcal{E}}^{2}\lesssim\|v\|_{\mathcal{E}}^{2}.
\end{equation}

A combination of (\ref{year25+}), (\ref{year24}), (\ref{year26}), and (\ref{year27}) yields
\begin{equation}\label{year28}
(\varepsilon^{2}\nabla(\psi-\psi_{h}),\nabla
v)+(\psi-\psi_{h},v)\leq\Big\{\sqrt{\gamma}\|\psi-\psi_{h}\|_{\mathcal{E}}+
C\displaystyle\Big(\sum\limits_{T\in\mathcal{T}_{h}}\eta_{\psi,T}^{2}\Big)^{1/2}\Big\}\|v\|_{\mathcal{E}}.
\end{equation}
From (\ref{year23}) and (\ref{year28}), we obtain
\begin{equation*}
\|\psi-\psi_{h}\|_{\mathcal{E}}\leq\sqrt{\gamma}\|\psi-\psi_{h}\|_{\mathcal{E}}+
C\displaystyle\Big(\sum\limits_{T\in\mathcal{T}_{h}}\eta_{\psi,T}^{2}\Big)^{1/2},
\end{equation*}
which leads to the desired estimate (\ref{year13}). Repeating the proof of (\ref{year11}) and (\ref{year12}),
we obtain (\ref{year14}) and (\ref{year15}).
\end{proof}
\begin{remark}
The condition
$-\varepsilon^{2}\triangle(\psi-\psi_{h})+(\psi-\psi_{h})\neq0$ is
usually satisfied, since
\begin{equation*}
-\varepsilon^{2}\triangle(\psi-\psi_{h})+(\psi-\psi_{h})=f-(-\varepsilon^{2}\triangle\psi_{h}+\psi_{h})
\end{equation*}
is the residual, which doesn't vanish in usual. Here
$\triangle\psi_{h}$ is the piecewise Laplacian of $\psi_{h}$.
\end{remark}

\section{The analysis of the efficiency on the estimators}
In this section, we analyze the efficiency of the {\it a posteriori} error estimates developed in Section 4. To
avoid the appearance of high order term, we assume that $f$ is a piecewise polynomial.
\begin{lemma}\label{year29}
For all $T\in\mathcal{T}_{h}$, there hold
\begin{equation}\label{year30}
\alpha_{T}\|f+\varepsilon^{2}\triangle\psi_{h}-\psi_{h}\|_{T}\lesssim\|\psi-\psi_{h}\|_{\mathcal{E},T}
\end{equation}
and
\begin{equation}\label{year31}
h_{T}\|\triangle
u_{h}+\psi_{h}\|_{T}\lesssim|u-u_{h}|_{1,T}+h_{T}\|\psi-\psi_{h}\|_{T}.
\end{equation}
\end{lemma}
\begin{proof}
We first prove (\ref{year30}). To this end, let $v=f+\varepsilon^{2}\triangle\psi_{h}-\psi_{h}$. Recall
the bubble function $\psi_{T}$ introduced in Section 3.
From (\ref{year4}), integration by parts, and (\ref{year6}), we have
\begin{equation*}
\begin{array}{lll}
\|v\|_{T}^{2}&\lesssim&(\psi_{T}v,v)_{T}\vspace{2mm}\\
&=&(-\varepsilon^{2}\triangle(\psi-\psi_{h}),\psi_{T}v)_{T}+(\psi-\psi_{h},\psi_{T}v)_{T}\vspace{2mm}\\
&=&\varepsilon^{2}(\nabla(\psi-\psi_{h}),\nabla(\psi_{T}v))_{T}+(\psi-\psi_{h},\psi_{T}v)_{T}\vspace{2mm}\\
&\leq&\|\psi-\psi_{h}\|_{\mathcal{E},T}\|\psi_{T}v\|_{\mathcal{E},T}\vspace{2mm}\\
&\lesssim&\|\psi-\psi_{h}\|_{\mathcal{E},T}\alpha_{T}^{-1}\|v\|_{T}.
\end{array}
\end{equation*}
The desired estimate (\ref{year30}) follows.

We next prove (\ref{year31}). For convenience, denote $v=\triangle u_{h}+\psi_{h}$. Similarly, we have
from $\psi=-\triangle u$ that
\begin{equation*}
\begin{array}{lll}
\|v\|_{T}^{2}&\lesssim&(\psi_{T}v,v)_{T}\vspace{2mm}\\
&=&(\triangle u_{h}-\triangle u+\triangle u+\psi_{h},\psi_{T}v)_{T}\vspace{2mm}\\
&=&-(\triangle(u-u_{h}),\psi_{T}v)_{T}-(\psi-\psi_{h},\psi_{T}v)_{T}\vspace{2mm}\\
&=&(\nabla(u-u_{h}),\nabla(\psi_{T}v))_{T}-(\psi-\psi_{h},\psi_{T}v)_{T}\vspace{2mm}\\
&\leq&|u-u_{h}|_{1,T}|\psi_{T}v|_{1,T}+\|\psi-\psi_{h}\|_{T}\|\psi_{T}v\|_{T}.
\end{array}
\end{equation*}
Applying inverse estimate and (\ref{year5}), we have
\begin{equation*}
\|v\|_{T}^{2}\lesssim(h_{T}^{-1}|u-u_{h}|_{1,T}+\|\psi-\psi_{h}\|_{T})\|v\|_{T}.
\end{equation*}
The estimate (\ref{year31}) follows immediately.
\end{proof}

\begin{lemma}\label{year32}
For all $E\in\mathcal{E}_{h}^{0}$, there hold
\begin{equation}\label{year33}
    \alpha_{E} \Big \|[\varepsilon^{2} \frac{\partial\psi_{h}}{\partial{\bf n}}]\Big\|_{E} \lesssim\|\psi-\psi_{h}\|_{\mathcal{E},\omega_{E}}
\end{equation}
and
\begin{equation}\label{year34}
    h_{E}^{1/2} \Big\| [\frac{\partial u_{h}}{\partial{\bf n}}] \Big\|_{E} \lesssim|u-u_{h}|_{1,\omega_{E}} + h_{E}\|\psi-\psi_{h}\|_{\omega_{E}}.
\end{equation}
\end{lemma}
\begin{proof}
We first prove (\ref{year33}). To this end,
let $\sigma=[\varepsilon^{2}\displaystyle\frac{\partial\psi_{h}}{\partial{\bf n}}]$.
Recall the bubble function $\psi_{E,\theta_{E}}$ and the extension operator $P_{E}$ introduced in Section 3.
Let $v_{E}=\psi_{E,\theta_{E}}P_{E}\sigma$. An application of integration by parts leads to
\begin{equation*}
\begin{array}{ll}
    &(\varepsilon^{2}\nabla(\psi-\psi_{h}),\nabla v_{E})_{\omega_{E}}+(\psi-\psi_{h},v_{E})_{\omega_{E}}\vspace{2mm}\\
    = &(f+\varepsilon^{2}\triangle_{h}\psi_{h}-\psi_{h},v_{E})_{\omega_{E}} + \Big(-[\varepsilon^{2}\displaystyle\frac{\partial\psi_{h}}{\partial{\bf n}}],v_{E} \Big)_{E},
\end{array}
\end{equation*}
where $\triangle_{h}$ is the elementwise Laplace operator. A combination of the above equality
and (\ref{year7})-(\ref{year9}) leads to
\begin{equation*}
\begin{array}{lll}
\|\sigma\|_{E}^{2}&\lesssim& \Big( [\varepsilon^{2}\displaystyle\frac{\partial\psi_{h}}{\partial{\bf n}}],v_{E} \Big)_{E}\vspace{2mm}\\
&=&(f+\varepsilon^{2}\triangle_{h}\psi_{h}-\psi_{h},v_{E})_{\omega_{E}}\vspace{2mm}\\
&\ &\ \ -(\varepsilon^{2}\nabla(\psi-\psi_{h}),\nabla v_{E})_{\omega_{E}}-(\psi-\psi_{h},v_{E})_{\omega_{E}}\vspace{2mm}\\
&\lesssim&\|f+\varepsilon^{2}\triangle_{h}\psi_{h}-\psi_{h}\|_{\omega_{E}}\|v_{E}\|_{\omega_{E}}+
\|\psi-\psi_{h}\|_{\mathcal{E},\omega_{E}}\|v_{E}\|_{\mathcal{E},\omega_{E}}\vspace{2mm}\\
&\lesssim&\varepsilon^{1/2}\min\{1,h_{E}\varepsilon^{-1}\}^{1/2}\|f+\varepsilon^{2}\triangle_{h}\psi_{h}-
\psi_{h}\|_{\omega_{E}}\|\sigma\|_{E}\vspace{2mm}\\
&\ &\ \
+\|\psi-\psi_{h}\|_{\mathcal{E},\omega_{E}}\varepsilon^{1/2}\min\{1,h_{E}\varepsilon^{-1}\}^{-1/2}\|\sigma\|_{E}.
\end{array}
\end{equation*}
By the definition of $\alpha_{E}$ for $E\subset\partial T$ and the local shape regularity of the mesh, we obtain
from the above inequality that
\begin{equation*}
\begin{array}{lll}
\alpha_{E}\|\sigma\|_{E}&\lesssim&\varepsilon^{1/2}\min\{1,h_{E}\varepsilon^{-1}\}^{1/2}
\varepsilon^{-1/2}\min\{1,h_{T}\varepsilon^{-1}\}^{1/2}\|f+\varepsilon^{2}\triangle_{h}\psi_{h}-
\psi_{h}\|_{\omega_{E}}\vspace{2mm}\\
&\ &\ \
+\|\psi-\psi_{h}\|_{\mathcal{E},\omega_{E}}\varepsilon^{1/2}\min\{1,h_{E}\varepsilon^{-1}\}^{-1/2}
\varepsilon^{-1/2}\min\{1,h_{T}\varepsilon^{-1}\}^{1/2}\vspace{2mm}\\
&\lesssim&\min\{1,h_{T}\varepsilon^{-1}\}\|f+\varepsilon^{2}\triangle_{h}\psi_{h}-
\psi_{h}\|_{\omega_{E}}+\|\psi-\psi_{h}\|_{\mathcal{E},\omega_{E}}\vspace{2mm}\\
&\lesssim&\|\psi-\psi_{h}\|_{\mathcal{E},\omega_{E}}.
\end{array}
\end{equation*}
In the last step, estimate (\ref{year30}) is used. We complete the
proof of (\ref{year33}).

We next prove (\ref{year34}). For convenience, denote $\sigma=[\displaystyle\frac{\partial u_{h}}{\partial{\bf n}}]$
and $v_{E}=\psi_{E}P_{E}\sigma$, where $\psi_{E}=\psi_{E,\theta_{E}}$ for $\theta_{E}=1$. Similarly, we have
\begin{equation*}
(\nabla(u-u_{h}),\nabla v_{E})_{\omega_{E}}=(-\triangle u+\triangle_{h}u_{h},v_{E})_{\omega_{E}}-
\Big( [\frac{\partial u_{h}}{\partial{\bf n}}],v_{E} \Big)_{E},
\end{equation*}
which leads to the following estimate:
\begin{equation*}
\begin{array}{lll}
\|\sigma\|_{E}^{2}&\lesssim& \Big( \displaystyle[\frac{\partial u_{h}}{\partial{\bf n}}],v_{E} \Big)\vspace{2mm}\\
&=&(\psi+\triangle_{h}u_{h},v_{E})_{\omega_{E}}-(\nabla(u-u_{h}),\nabla v_{E})_{\omega_{E}}\vspace{2mm}\\
&=&(\psi-\psi_{h},v_{E})_{\omega_{E}}+(\triangle_{h}u_{h}+\psi_{h},v_{E})_{\omega_{E}}-(\nabla(u-u_{h}),\nabla v_{E})_{\omega_{E}}\vspace{2mm}\\
&\lesssim&\|\psi-\psi_{h}\|_{\omega_{E}}h_{E}^{1/2}\|\sigma\|_{E}+h_{E}^{1/2}\|\triangle_{h}u_{h}
+\psi_{h}\|_{\omega_{E}}\|\sigma\|_{E}\vspace{2mm}\\
&\ &\ +|u-u_{h}|_{1,\omega_{E}}h_{E}^{-1/2}\|\sigma\|_{E}.
\end{array}
\end{equation*}
We obtain from the above inequality that
\begin{equation*}
\begin{array}{lll}
h_{E}^{1/2}\|\sigma\|_{E}&\lesssim&h_{E}\|\psi-\psi_{h}\|_{\omega_{E}}+h_{E}\|\triangle_{h}u_{h}
+\psi_{h}\|_{\omega_{E}}+|u-u_{h}|_{1,\omega_{E}}\vspace{2mm}\\
&\lesssim&h_{E}\|\psi-\psi_{h}\|_{\omega_{E}}+|u-u_{h}|_{1,\omega_{E}}.
\end{array}
\end{equation*}
In the last step above, we employ the estimate (\ref{year31}). We complete the proof of (\ref{year34}).
\end{proof}

\begin{theorem}
Let $(\psi,u)\in H_{0}^{1}(\Omega)\times H_{0}^{1}(\Omega)$ and
$(\psi_{h},u_{h})\in V_{h}^{0}\times V_{h}^{0}$ be the solutions to
(\ref{discrete mixed formulation P1}) and (\ref{mixed from P1}),
respectively. Then there exist positive constants $C_{7}$ and
$C_{8}$, independent of the mesh-size function $h$ and
$\varepsilon$, such that
\begin{equation*}
\displaystyle
C_{7}\Big\{\sum\limits_{T\in\mathcal{T}_{h}}\eta_{\psi,T}^{2}\Big\}^{1/2}\leq\|\psi-\psi_{h}\|_{\mathcal{E}}
\end{equation*}
and
\begin{equation*}
\displaystyle
C_{8}\Big\{\sum\limits_{T\in\mathcal{T}_{h}}\eta_{\psi,T}^{2}+\eta_{u,T}^{2}\Big\}^{1/2}\leq\|(\psi-\psi_{h},u-u_{h})\|.
\end{equation*}
\end{theorem}
\begin{proof}
Summing (\ref{year30}) and (\ref{year33}) over all $T\in\mathcal{T}_{h}$, we obtain the first estimate. Similarly,
we get the second one.
\end{proof}

\begin{theorem}
Let $(\psi,u)\in H^{1}(\Omega)\times H_{0}^{1}(\Omega)$ and
$(\psi_{h},u_{h})\in V_{h}\times V_{h}^{0}$ be the solutions to
(\ref{discrete mixed formulation P2}) and (\ref{mixed form P2}),
respectively. Then there exist positive constants $C_{9}$ and
$C_{10}$, independent of the mesh-size function $h$ and
$\varepsilon$, such that
\begin{equation*}
C_{9}\displaystyle\Big\{\sum\limits_{T\in\mathcal{T}_{h}}\eta_{\psi,T}^{2}\Big\}^{1/2}\leq\|\psi-\psi_{h}\|_{\mathcal{E}}
\end{equation*}
and
\begin{equation*}
C_{10}\displaystyle\Big\{\sum\limits_{T\in\mathcal{T}_{h}}\eta_{\psi,T}^{2}+\eta_{u,T}^{2}\Big\}^{1/2}\leq\|(\psi-\psi_{h},u-u_{h})|\|.
\end{equation*}
\end{theorem}
\begin{proof}
These two estimates follow from Lemmas \ref{year29} and \ref{year32}.
\end{proof}

\begin{remark}
Theorems 4.1 and 5.3 (and Theorem 4.2 and 5.4) indicate that the
ratios between the upper and lower bounds, i.e., $C_{1}/C_{7}$ and
$C_{3}/C_{8}$ (and $C_{4}/C_{9}$ and $C_{6}/C_{10}$), do not depend
on the singular perturbation parameter $\varepsilon$. Therefore, the
estimators developed in this paper are fully robust with respect to
$\varepsilon$. This further implies that each component of the new
measure of the error is balanced with respect to the perturbation
parameter.
\end{remark}

\section{Numerical experiments}
In this section, we test our {\it a posteriori} error estimators on
two model problems. Note that all programs were developed by
ourselves.
\subsection{Example one}
Consider problem (\ref{PDE1}) and (\ref{Boundary condition2}) on the unit square $\Omega=(0,1)\times(0,1)$.
We suppose the exact solution of this model has the form
\begin{equation*}
u(x,y)=\displaystyle256(x^{2}+\varepsilon^{2}(1-\exp(-{x}/{\varepsilon}))^{2})(x-1)^{2}y^{2}(y-1)^{2}.
\end{equation*}
The function $u$ has a boundary layer, which varies significantly near $x=0$.

\begin{figure}[t]
    \centering
    \includegraphics[width=4.5in]{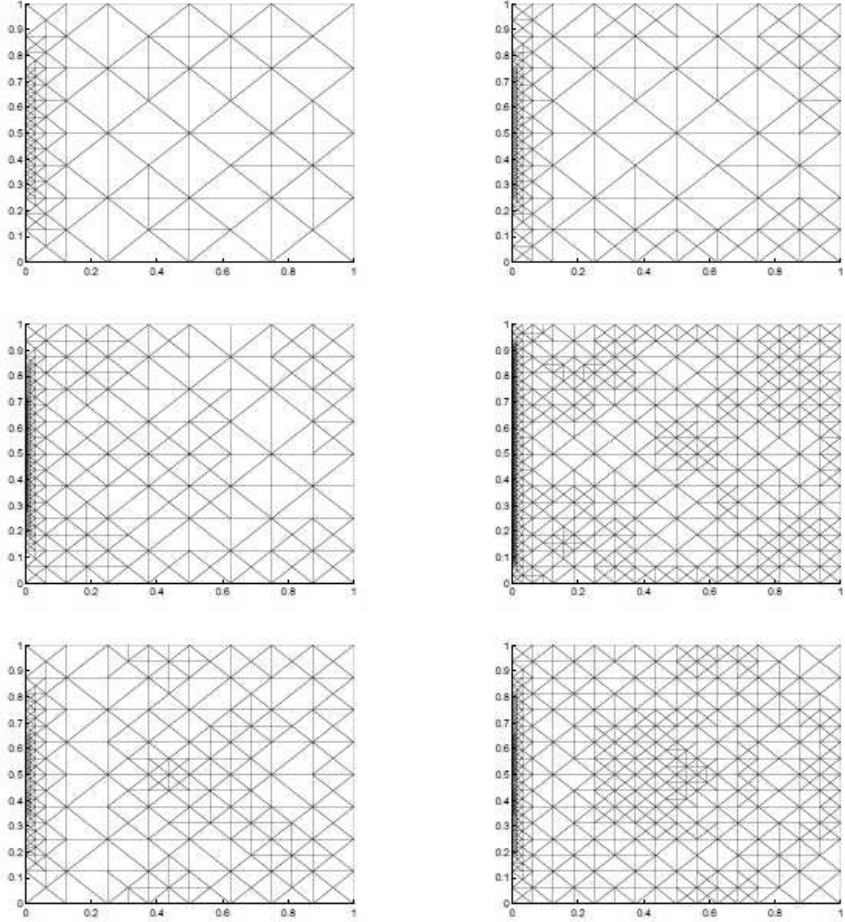}
    \caption{Upper: The mesh after 10 iterations with 214 triangles (left) and the mesh after 12 iterations with 402 triangles (right). Middle: The mesh after 14 iterations with 727 triangles (left) and the mesh after 16 iterations with 4470 triangles (right). These four plots depict the elementwise indicator $\eta_{\psi,T}$. Lower: The mesh after 10 iterations with 478 triangles (left) and the mesh after 12 iterations with 1010 triangles (right), generated by the elementwise indicator $\eta_{\psi,T}+\eta_{u,T}$. Here $\varepsilon=10^{-6}$ and $\theta=0.3$ for all plots.}\label{Fig1}
\end{figure}

Our initial mesh consists of eight isosceles right triangles. We
employ D\"{o}rfler marking strategy \cite{Dorfler} with the marking
parameter $\theta=0.3$ and use the ``longest edge" refinement
\cite{Rivara1} to obtain an admissible mesh. Plots in
Figure~\ref{Fig1} depict the estimators of
$\|\psi-\psi_{h}\|_{\mathcal{E}}=\big(\varepsilon^{2}\|\nabla(\psi-\psi_{h})\|^{2}+\|\psi-\psi_{h}\|^{2}
\big)^{1/2}$ (upper and middle), and
$\|(\psi-\psi_{h},u-u_{h})\|=\big(\|\psi-\psi_{h}\|_{\mathcal{E}}^{2}+\|\nabla(u-u_{h})\|^{2}
\big)^{1/2}$ (lower), respectively. We observe that strong mesh
refinement near the line $x=0$, which indicates the estimators of
the errors $\|\psi-\psi_{h}\|_{\mathcal{E}}$ and
$\|(\psi-\psi_{h},u-u_{h})\|$ capture boundary layers well.

\begin{figure}[t]
    \centering
    \includegraphics[width=5in]{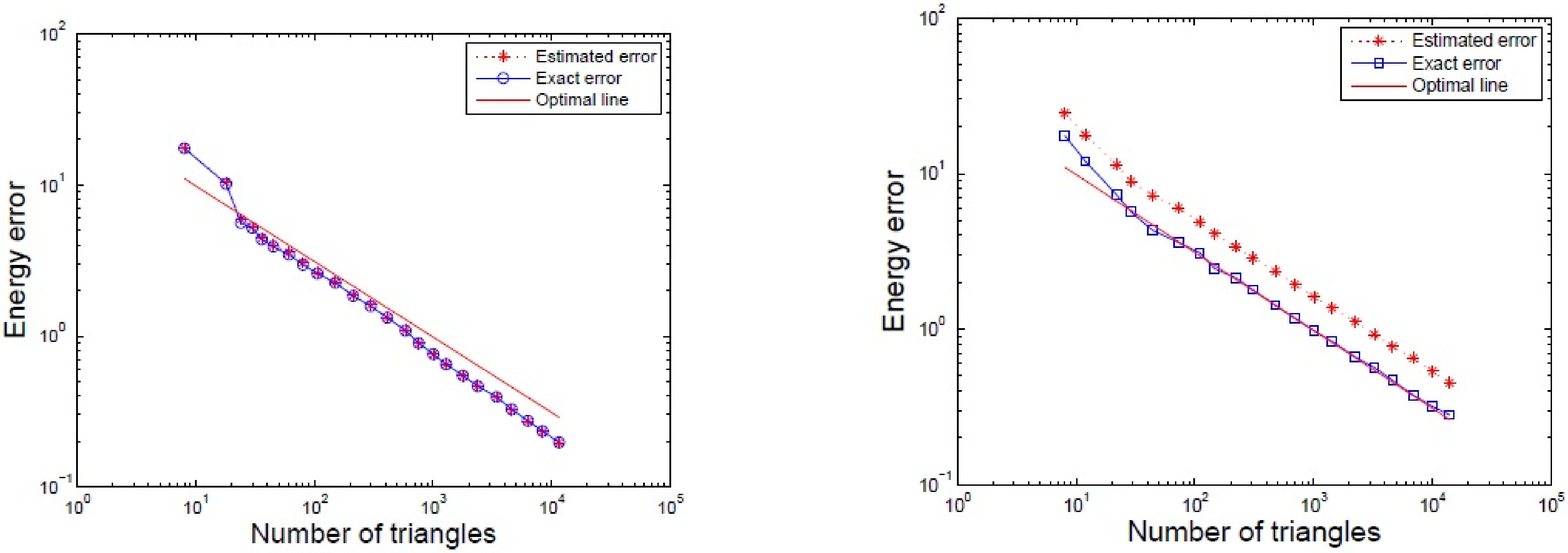}
    \caption{Approximations for $u$ (left) and  $\psi=-\triangle u$ (right) on an adaptively refined
    mesh with 6094 triangles, which are generated by the elementwise indicator $\eta_{\psi,T}$.
    Here $\varepsilon=10^{-6}$ and $\theta=0.3$.}\label{Fig4}
\end{figure}

Figure \ref{Fig4} demonstrates finite element approximations to $u$
(left) and $\psi=-\triangle u$ (right). It is observed that the
function $u$ doesn't possess layer, and that $\psi$ has boundary
layer near $x=0$. On the other hand, the upper two plots of Figure
\ref{Fig6} display the estimated and exact errors for
$\|\psi-\psi_{h}\|_{\mathcal{E}}$ (left) and
$\|(\psi-\psi_{h},u-u_{h})\|$ (right), respectively. It is observed
that the estimated convergence curve overlaps the curve of
$\|\psi-\psi_{h}\|_{\mathcal{E}}$, which indicates that the
estimator for $\|\psi-\psi_{h}\|_{\mathcal{E}}$ is asymptotically
exact even for very small $\varepsilon$. We also observe that the
estimated convergence curve is parallel to the curve
$\|(\psi-\psi_{h},u-u_{h})\|$ independent of $\varepsilon=10^{-6}$,
and both curves decrease in optimal rates. Note that the
study of convergence and optimality of adaptive algorithms is
still in its infancy, and has been carried out mainly for standard
adaptive finite element method for general second order elliptic
problems; see, e.g., \cite{BW85, BV96, Carstensen;Rabus, Chen;Holst;Xu, CN00, DK16, Du, Morin;Nochetto;Siebert}.
\begin{figure}[t]
    \centering
    \includegraphics[width=5in]{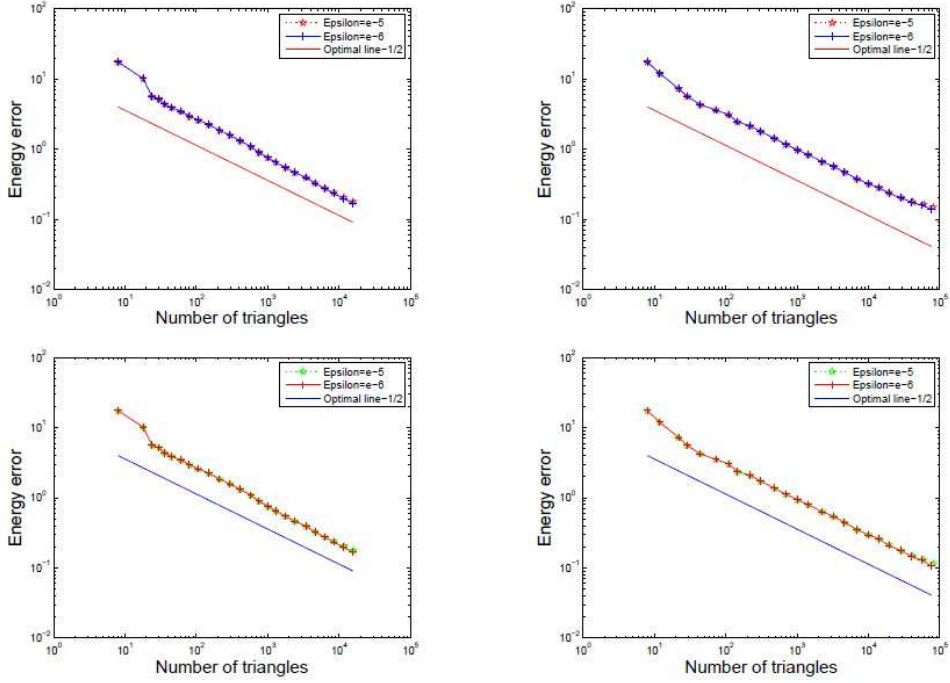}
    \caption{Upper: Estimated and exact errors of $\|\psi-\psi_{h}\|_{\mathcal{E}}$ (left) and $\|(\psi-\psi_{h},u-u_{h})\|$ (right) against
    the number of elements in adaptively refined meshes for
    $\varepsilon=10^{-6}$. Lower: Exact errors of $\|\psi-\psi_{h}\|_{\mathcal{E}}$ (left) and
    $\|(\psi-\psi_{h},u-u_{h})\|$ (right)
    against the number of elements in adaptively refined meshes for $\varepsilon=10^{-5}$ and $\varepsilon=10^{-6}$. Here
    the marking parameter $\theta=0.3$.}\label{Fig6}
\end{figure}

The two lower plots of Figure \ref{Fig6} depict error curves for
$\|\psi-\psi_{h}\|_{\mathcal{E}}$ (left) and
$\|(\psi-\psi_{h},u-u_{h})\|$ (right), respectively. It is observed
that the convergence curves for $\varepsilon=10^{-5}$ and
$\varepsilon=10^{-6}$ are consistent, which indicates that the
errors reduce uniformly with respect to $\varepsilon$. In addition,
we include in Figure \ref{Fig6} an optimal theoretical convergence line with slope $-1/2$.
The plots indicate that $\|\psi-\psi_{h}\|_{\mathcal{E}}$ and
$\|(\psi-\psi_{h},u-u_{h})\|$ decrease in the optimal convergence
rates.

Tables \ref{Tab:EgOne1} and \ref{Tab:Egone2} show some
results of the actual errors $\|\psi-\psi_{h}\|_{\mathcal{E}}$ and
$\|(\psi-\psi_{h},u-u_{h})\|$, the {\it a posteriori} indicators
$\eta_{\psi}$ and $\eta_{\psi}+\eta_{u}$, and the effectivity
indexes eff-index$_{\psi}$ for $\psi$ and eff-index$_{\psi+u}$ for
$(\psi,u)$ for Example 1, where
eff-index$_{\psi}=\eta_{\psi}/\|\psi-\psi_{h}\|_{\mathcal{E}}$,
eff-index$_{\psi+u}=(\eta_{\psi}+\eta_{u})/\|(\psi-\psi_{h},u-u_{h})$. 
It is observed that the effectivity indices of the error $\|\psi-\psi_{h}\|_{\mathcal{E}}$ 
are close to 1, and that the effectivity indices of the error
$\|\psi-\psi_{h}\|_{\mathcal{E}}$ are about 1.5. This suggests that
our estimators are robust with respect to $\varepsilon$.

\begin{table}[t]\small
 \begin{center}
        \caption{Example 1: $k$ -- number of iterations; $\eta_{\psi}$ -- numerical result of estimated error
        for $\|\psi-\psi_{h}\|_{\mathcal{E}}$; eff-index$_{\psi}$ -- the corresponding effectivity index for $\psi$ (the ratio of estimated
        and exact errors).
        Here $\varepsilon=10^{-5}$, $\theta=0.5$.} \label{Tab:EgOne1}
        \small 
        \begin{tabular}{|c|c|c|c|c|c|c|c|c|} \hline
            $k$& $2$& $4$& $6$& $8$& $10$& $12$& $14$& $16$\\ \hline
            $\|\psi-\psi_{h}\|_{\mathcal{E}}$&10.170&5.1837&3.6752&2.1023&1.1593&0.6290&0.3374&0.1911\\ \hline
            $\eta_{\psi}$&10.400&5.2507&3.7688&2.1295&1.1557&0.6286&0.3336&0.2158\\ \hline
            eff-index$_{\psi}$&1.0226&1.0129&1.0249&1.0130&0.9969&0.9914&0.9887&1.1290\\ \hline
        \end{tabular}
    \end{center}
\end{table}
\begin{table}[t]\small
 \begin{center}
        \caption{Example 1: $\eta_{\psi}+\eta_{u}$ -- numerical result of estimated error
        for $\|(\psi-\psi_{h},u-u_{h})\|$ (is denoted by err$_{\psi+u}$); eff-index$_{\psi+u}$ -- the
        corresponding effectivity index for $(\psi,u)$.
        Here $\varepsilon=10^{-5}$, $\theta=0.5$.}
        \label{Tab:Egone2}
        \small 
        \begin{tabular}{|c|c|c|c|c|c|c|c|c|} \hline
            $k$& $1$& $3$& $5$& $7$& $9$& $11$& $13$& $15$\\ \hline
            err$_{\psi+u}$&17.559&7.3066&3.6330&1.9267&0.9408&0.5028&0.3040&0.1808\\ \hline
            $\eta_{\psi}+\eta_{u}$&24.481&11.2702&5.9308&2.9577&1.5630&0.8571&0.4748&0.2999\\ \hline
            eff-index$_{\psi+u}$&1.3942&1.5425&1.6325&1.5351&1.6613&1.7045&1.5617&1.6583\\ \hline
        \end{tabular}
    \end{center}
\end{table}

\subsection{Example two}
This model is taken from \cite{Han}. Consider (\ref{PDE1})-(\ref{Boundar condition1}) on the unit square
$\Omega=(0,1)\times(0,1)$ with the source term
\begin{equation*}
f(x,y)=2\pi^{2}(1-\cos2\pi x\cos2\pi y).
\end{equation*}
Although the exact solution of this model problem is unknown, we know that the exact solution $u$ has
four sharp boundary layers near the boundary.

\begin{figure}[t]
    \centering
    \includegraphics[width=4.5in]{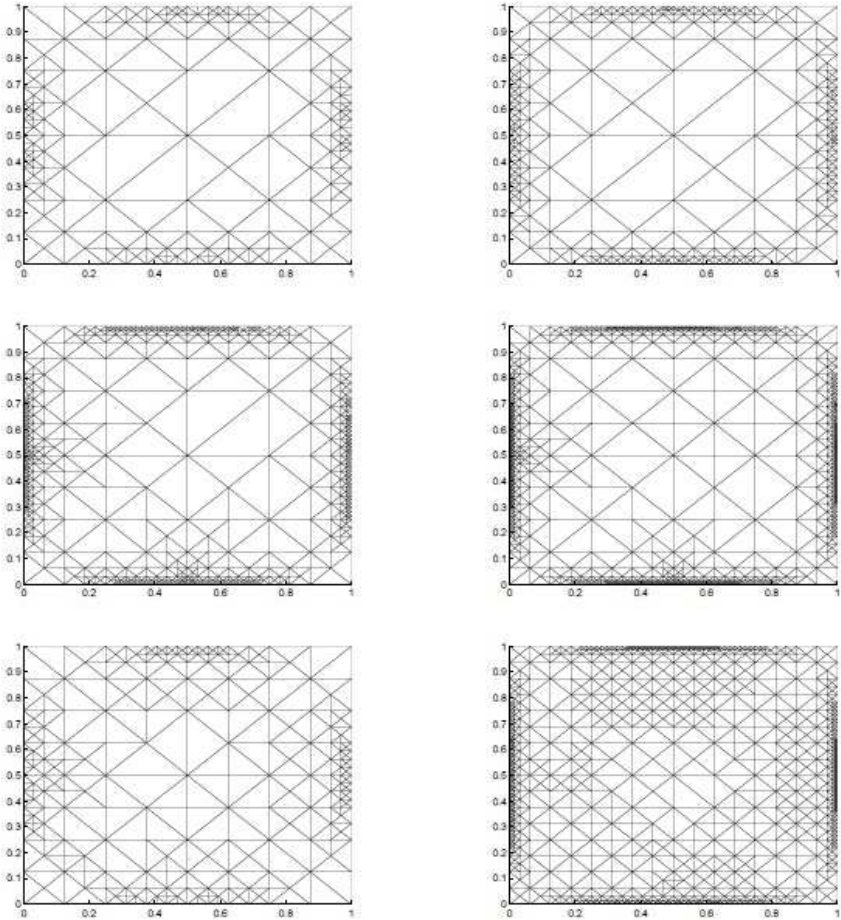}
    \caption{Upper: The mesh after 10 iterations with 344 triangles (left) and the mesh
    after 12 iterations with 611 triangles (right), generated by the elementwise
    indicator $\eta_{\psi,T}$. Middle: The mesh after 14 iterations with 1136 triangles (left) and the mesh
    after 16 iterations with 2109 triangles (right), generated by
    the elementwise indicator $\eta_{\psi,T}$. Lower: The mesh after 8 iterations with 418 triangles (left) and the mesh
    after 12 iterations with 1936 triangles (right), generated by
    the elementwise indicator $\eta_{\psi,T}+\eta_{u,T}$. Here $\varepsilon=10^{-7},\theta=0.3$.}\label{Fig9}
\end{figure}

We choose the same initial mesh as in Example one and set the
marking parameter $\theta=0.3$. The upper and middle four plots
of Figure \ref{Fig9} show the mesh generated by the estimator of
$\|\psi-\psi_{h}\|_{\mathcal{E}}$ after 10, 12, 14, and
16 iterations, and the lower two plots show the mesh by the estimator of $\|(\psi-\psi_{h},u-u_{h})\|$ after 8 and
12 iterations. It is observed that the estimators of
$\|\psi-\psi_{h}\|_{\mathcal{E}}$ and $\|(\psi-\psi_{h},u-u_{h})\|$
capture the layers well, and that the refinement concentrates around
four sharp boundary layers. This indicates that our estimators recognize
the behavior of the solution well, even when the singularly
perturbed parameter is very small.

Figure \ref{Fig10} reports the finite element approximation to $u$
(left) and $\psi=-\triangle u$ (right), respectively. Notice that
the immediate variable $\psi$ has four sharp boundary layers, and
that the primal variable $u$ does not have layer.
\begin{figure}[t]
    \centering
    \includegraphics[width=5in]{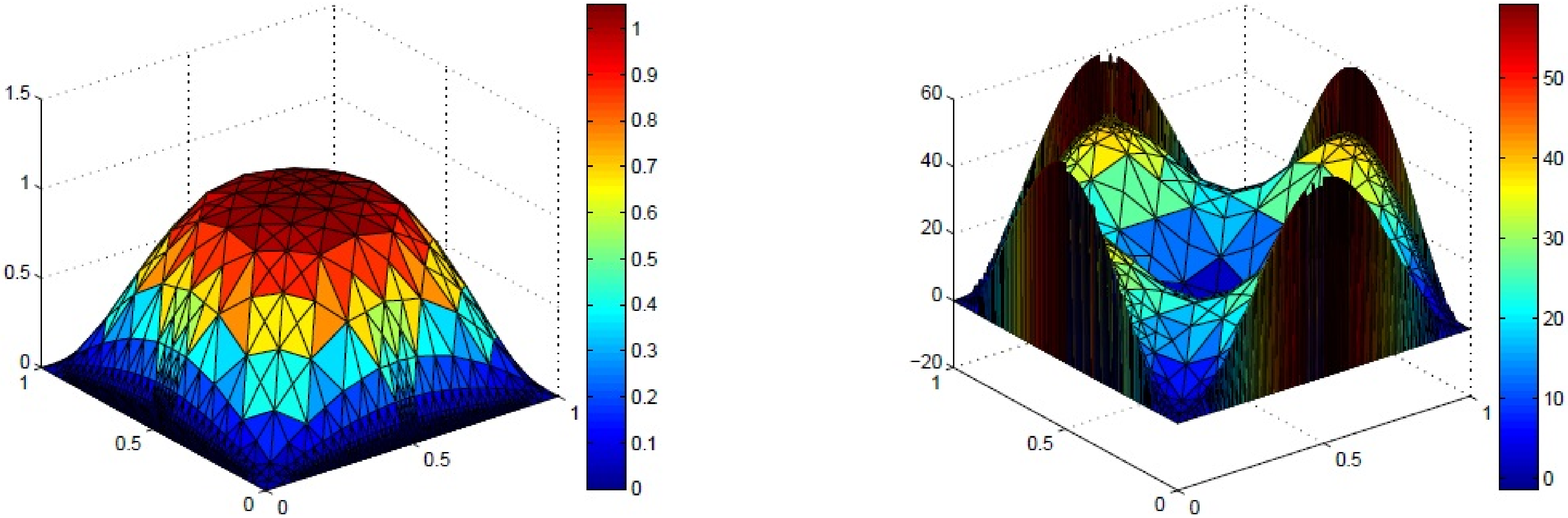}
    \caption{Approximations for $u$ (left) and
    $\psi=-\triangle u$ (right), respectively, on an adaptively
    refined mesh with 9355 triangles, which are generated by
    the elementwise indicator $\eta_{\psi,T}$. Here $\varepsilon=10^{-7}$ and $\theta=0.3$.}\label{Fig10}
\end{figure}
\begin{figure}[t]
    \centering
    \includegraphics[width=5in]{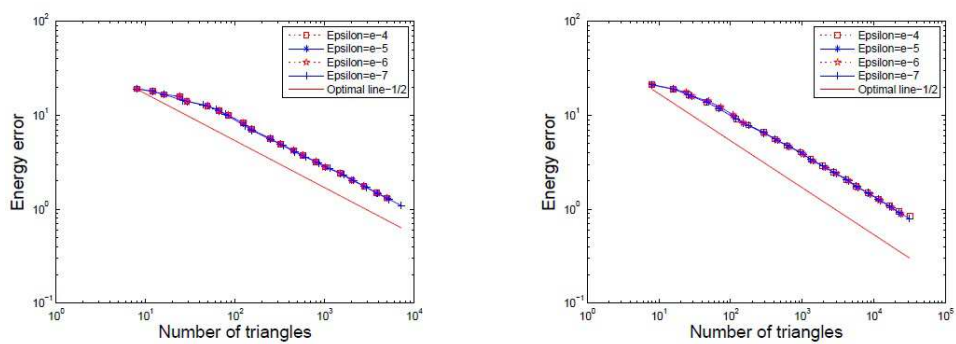}
    \caption{$\|\psi-\psi_{h}\|_{\mathcal{E}}$ (left) and
    $\|(\psi-\psi_{h},u-u_{h})\|$ (or $|u-u_{h}|_{1}$) (right) against
    the number of elements in adaptively refined meshes for $\varepsilon$
    from $\varepsilon=10^{-4}$ to $\varepsilon=10^{-7}$, where the marking parameter $\theta=0.3$.}\label{Fig11}
\end{figure}

\begin{table}[t]\small
       \begin{center}
        \caption{Example 2: TOL -- given tolerance, $k$ -- number of iterations; $\eta_{k}$ -- numerical result of estimated error for
$\|\psi-\psi_{h}\|_{\mathcal{E}}$, DOF-- degrees of freedom,
        $h_{\rm min}(\varepsilon)$-- smallest mesh size.
        Here $\varepsilon=10^{-5}$ and $\theta=0.4$.}
        \label{Egtwo1}
        \small 
        \begin{tabular}{|c|c|c|c|c|c|c|c|} \hline
            ${\rm TOL}$& $20$& $10$& $5$& $2.5$& $1.25$& $0.625$& $0.3125$\\ \hline
            $k$& $1$& $6$& $9$& $12$& $15$& $19$& $22$\\ \hline
            $\eta_{k}$&19.200&9.9911&4.0486&2.2030&1.1698& $0.5093$& $0.2754$\\ \hline
            ${\rm DOF}$&9&56&278&1041&5243&$19062$&$67485$\\ \hline
            $h_{\rm min}(\varepsilon)$&0.500&0.1250&1.56e-02&3.91e-03&1.38e-03&$2.44e-04$& $6.10e-05$\\ \hline
        \end{tabular}
       \end{center}
\end{table}

Table \ref{Egtwo1} reports the given tolerance TOL,
the number of iterations $k$, the estimated error ($\eta_{k}$) for
$\|\psi-\psi_{h}\|_{\mathcal{E}}$, the degrees of freedom DOF, the
smallest mesh size $h_{\rm min}(\varepsilon)$ for example 2, which show that the
required DOF depends on both TOL and $\varepsilon$, and that the layer
is gradually resolved, because the smallest mesh size $h_{\rm
min}(\varepsilon)$ has arrived at the magnitude of $\varepsilon$
after 22 iterations.

Figure \ref{Fig11} shows the estimated errors of
$\|\psi-\psi_{h}\|_{\mathcal{E}}$ (left) and
$\|(\psi-\psi_{h},u-u_{h})\|$ (or $|u-u_{h}|_{1}$) (right),
respectively. We observe again that the estimated errors reduce
uniformly with respect to $\varepsilon$ in both norms with almost
optimal rate $-1/2$.

\end{document}